\documentclass[12pt,reqno]{amsart}
\usepackage{amsmath,amssymb,amsthm,mathtools,calc,verbatim,enumitem,tikz,url,mathrsfs,cite,fullpage,hyperref}

\newtheorem{theorem}{Theorem}[section]
\newtheorem{lemma}[theorem]{Lemma}
\newtheorem{corollary}[theorem]{Corollary}

\theoremstyle{definition}
\newtheorem{defn}[theorem]{Definition}
\newtheorem*{qu*}{Question}
\theoremstyle{remark}
\newtheorem{remark}[theorem]{Remark}

\newcommand\N{\mathbb{N}}

\newcommand\Z{\mathbb{Z}}
\newcommand\E{\operatorname{\mathbb{E}}}
\newcommand\cA{\mathcal{A}}

\newcommand\cH{\mathcal{H}}

\renewcommand\Pr{\operatorname{\mathbb{P}}}
\newcommand\id{\hbox{$1\mkern-6.5mu1$}}

\newcommand\eps{\varepsilon}

\renewcommand\le{\leqslant}
\renewcommand\ge{\geqslant}
\renewcommand\to{\rightarrow}
\newcommand\ds{\displaystyle}

\begin{document}

\title{Erd\H{o}s covering systems}
\author{Paul Balister \and B\'ela Bollob\'as \and Robert Morris \and \\
Julian Sahasrabudhe \and Marius Tiba}


\address{Mathematical Institute, University of Oxford, Radcliffe Observatory Quarter, Woodstock Road, Oxford, OX2 6GG, UK}\email{Paul.Balister|marius.tiba@maths.ox.ac.uk}

\address{Department of Pure Mathematics and Mathematical Statistics, Wilberforce Road, Cambridge, CB3 0WA, UK, and Department of Mathematical Sciences,
University of Memphis, Memphis, TN 38152, USA}\email{bb12@cam.ac.uk}

\address{IMPA, Estrada Dona Castorina 110, Jardim Bot\^anico,
Rio de Janeiro, 22460-320, Brazil}\email{rob@impa.br}

\address{Department of Pure Mathematics and Mathematical Statistics, Wilberforce Road,
Cambridge, CB3 0WA, UK}\email{jdrs2@cam.ac.uk}

\thanks{The first two authors were partially supported by NSF grant DMS 11855745.}


\begin{abstract}
A \emph{covering system} is a finite collection of arithmetic progressions whose union is the set of integers. The study of these objects was initiated by Erd\H{o}s in 1950, and over the following decades he asked many questions about them. Most famously, he asked whether there exist covering systems with distinct moduli whose minimum modulus is arbitrarily large. This problem was resolved in 2015 by Hough, who showed that in any such system the minimum modulus is at most $10^{16}$.

The purpose of this note is to give a gentle exposition of a simpler and stronger variant of Hough's method, which was recently used to answer several other questions about covering systems. We hope that this technique, which we call the \emph{distortion method}, will have many further applications in other combinatorial settings.
\end{abstract}

\maketitle

\section{Introduction}

We say that a finite collection $\{A_1,\ldots,A_k\}$ of arithmetic progressions is a \emph{covering system} if $\bigcup_{i = 1}^k A_i = \Z$, that is, if their union covers the integers. The study of covering systems with distinct moduli (common differences) was initiated almost 70 years ago by Erd\H{o}s~\cite{E50}, who posed many problems about these systems over the following decades. The most famous of these was his so-called `minimum modulus problem', which asked whether there exist such systems with arbitrarily large minimum modulus. This problem was resolved by Hough~\cite{H} in 2015, following important work by Filaseta, Ford, Konyagin, Pomerance and Yu~\cite{FFKPY}.

\begin{theorem}[Hough, 2015]\label{thm:Hough}
In any covering system of the integers with distinct moduli, the minimum modulus is at most $10^{16}$.
\end{theorem}

\enlargethispage*{\baselineskip}
\thispagestyle{empty}

Hough's paper moreover introduced a new method, which we call the \emph{distortion method}. In this method, we reveal the progressions in stages, and define a sequence of probability measures, each of which depends only on the progressions revealed up to that point. These measures concentrate on the set of uncovered points, and allow us to maintain a constant lower bound on the measure of this set (which may be very small in the uniform measure).

The purpose of this note is to give a gentle introduction to a simpler and more powerful variant of Hough's method, which was introduced by the authors in two recent papers~\cite{BBMST1,BBMST2}. We will illustrate this method by giving a simple proof of Hough's theorem in the case of square-free moduli.\footnote{We emphasize that this proof can be extended to prove Theorem~\ref{thm:Hough} without much difficulty (see~\cite{BBMST1}), but this requires some tedious calculations that, for the sake of clarity, we wish to avoid.} Our aim is to make this method more widely known amongst the combinatorial community, in the hope that further applications will be discovered.

\section{A geometric setting}

For the purposes of exposition, it will be convenient to work in a (slightly more general) geometric setting. Let $S_1,\ldots,S_n$ be finite sets with at least two elements, and set
$$Q = S_1 \times \cdots \times S_n.$$
A {\em hyperplane} in $Q$ is a set $A = Y_1 \times \cdots \times Y_n \subset Q$, with each $Y_i$ either equal to $S_i$ or a singleton element of $S_i$, and the set of \emph{fixed coordinates} of $A$ is
$$F(A) := \big\{ k : Y_k \ne S_k \big\}.$$
We say that two hyperplanes $A$ and $A'$ are \emph{parallel\/} if $F(A) = F(A')$.

The following theorem was proved in~\cite{BBMST2}; we will give the proof in Sections~\ref{sec:distortion}--\ref{sec:proofGMM}, below.

\begin{theorem}\label{thm:GMM}
For every sequence of finite sets\/ $(S_k)_{k \ge 1}$ such that $|S_k| \ge 2$ for each $k \in \N$ and
\begin{equation}\label{eq:GMMcondition}
\liminf_{k \to \infty} \frac{|S_k|}{k} > 3,
\end{equation}
there exists a constant $C$ such that the following holds. Let\/ $\cA$ be a collection of hyperplanes that cover $Q := S_1 \times \dots\times S_n$ for some $n \in \N$. Then either two of the hyperplanes are parallel, or there exists a hyperplane\/ $A \in \cA$ with $F(A) \subset \{1,\ldots,C\}$.
\end{theorem}

Before continuing, let us note that Theorem~\ref{thm:GMM} implies Hough's theorem for covering systems with square-free moduli.

\begin{corollary}\label{cor:MM}
In any covering system of the integers with distinct square-free moduli, the minimum modulus is bounded by an absolute constant.
\end{corollary}

\begin{proof}
Simply apply Theorem~\ref{thm:GMM} with $S_k = \{1,\ldots, p_k\}$ for each $k \in \N$, where $p_1 < p_2 < \cdots$ are the prime numbers, listed in increasing order. To spell out the details, let $\cA$ be a covering system of the integers with distinct square-free moduli, let $p_n$ be the largest prime that divides one of the moduli, and set $Q := S_1 \times \cdots \times S_n$. Now, by the Chinese Remainder Theorem, each arithmetic progression $A = a + d\Z \in \cA$ corresponds to the hyperplane $Y_1 \times \cdots \times Y_n \subset Q$, where $Y_k = \{ a \pmod {p_k} \}$ if $p_k$ divides $d$, and $Y_k = S_k$ otherwise. We may therefore map $\cA$ into a finite collection $\cH$ of hyperplanes that covers $Q$, and since the moduli of $\cA$ are distinct, the hyperplanes in $\cH$ are non-parallel.

Now, by Theorem~\ref{thm:GMM}, there exists an arithmetic progression $A = a + d\Z \in \cA$ such that the set of fixed coordinates of the corresponding hyperplane is contained in $\{1,\ldots,C\}$ (where $C$ is the constant given by the theorem). But this means that $d$ divides (and hence at most) $p_1 \cdots p_C$, which is an absolute constant, as required.
\end{proof}

Note that $p_k \sim k \log k$, whereas in Theorem~\ref{thm:GMM} we allow the size of the sets $S_k$ to grow only linearly. We showed in~\cite{BBMST2} that Theorem~\ref{thm:GMM} is close to best possible, since there exists a sequence with $|S_k| \sim k$ for which the conclusion of the theorem fails.

In the next section we will give an overview of the distortion method, and prove a general lemma regarding covering. In Section~\ref{sec:moments} we will perform a simple moment calculation, and in Section~\ref{sec:proofGMM} we will deduce Theorem~\ref{thm:GMM}.

\section{The Distortion Method}\label{sec:distortion}

In this section we will give an outline of the proof of Theorem~\ref{thm:GMM}. We will work in the following general setting: let $S_1,\ldots,S_n$ be finite sets with at least two elements, set
$$Q := S_1 \times \cdots \times S_n,$$
and let $\cA$ be a collection of hyperplanes in $Q$. Our task is to show that if $|S_k|$ grows sufficiently quickly, and $\cA$ does not contain parallel hyperplanes, then $\cA$ cannot cover $Q$.

To do so, we will reveal the elements of $\cA$ in $n$ rounds, corresponding to the $n$ sets $S_1,\ldots,S_n$, and define a sequence of probability measures $\Pr_0, \ldots, \Pr_n$ on $Q$ that gradually distort the space. The measure $\Pr_k$ will depend on the elements of $\cA$ that were revealed in the first $k$ rounds, and will be chosen so that the $\Pr_k$-measure of the set covered in the $k$th round is small. However, it will be important that we do not change the measure of the set of points that were covered earlier, and we do not increase the measure of any set too much.

In order to define these measures, recall that $F(A)$ is the set of fixed coordinates of a hyperplane $A$, and define
$$\cA_k := \big\{ A \in \cA : \max(F(A)) = k \big\}$$
to be the set of hyperplanes that we reveal in round $k$, and
\begin{equation*}\label{def:B_i}
B_k := \bigcup_{A \in \cA_k} A
\end{equation*}
to be the set that is covered by those hyperplanes. Note that, since $F(A) \subset [k] = \{1,\ldots,k\}$ for every $A \in \cA_k$, we can consider $B_k$ to be a subset of
$$Q_k := S_1 \times \cdots \times S_k$$
by identifying $X \subset Q_k$ with $X \times S_{k+1} \times \dots \times S_n$. We call a set of this form $Q_k$-{\em measurable}.

Let $\Pr_0$ be the uniform probability measure on $Q$, and let us think of this as being the trivial measure on $Q_0$, the empty product. Let $1 \le k \le n$, and suppose that we have already defined a probability measure $\Pr_{k-1}$ on $Q_{k-1}$ (which we extend uniformly to a measure on~$Q$). A natural way (cf.~\cite{H}) to define the measure $\Pr_k$ on $Q_k$ would be to set $\Pr_k(B_k) = 0$, and redistribute the removed measure over the remaining elements (taking care not to change the measure of any $Q_{k-1}$-measurable set). However, it turns out to be helpful to define the measure  $\Pr_k$ in the following, slightly more subtle way.

Recall that $Q_k = Q_{k-1} \times S_k$, so the elements of $Q_k$ can be written as pairs $(x,y)$, where $x \in Q_{k-1}$ and $y \in S_k$. Now, for each $x \in Q_{k-1}$, define
\begin{equation}\label{def:alpha}
 \alpha_k(x) := \frac{\big| \big\{ y \in S_k : (x,y)\in B_k \big\} \big|}{|S_k|}, 
\end{equation}
that is, the proportion of the `fibre' $F_x := \{ (x,y) : y \in S_k \} \subset Q_k$ that is covered in round~$k$. Now, for some $\delta \in [0,1/2]$, we do one of two things on the fibre $F_x$, depending on whether or not $\alpha_k(x) \le \delta$:

\pagebreak

\begin{itemize}
\item If $\alpha_k(x) \le \delta$, then we set $\Pr_k(x,y) = 0$ for every element of $F_x \cap B_k$, and increase the measure proportionally on the rest of $F_x$;
\item If $\alpha_k(x) > \delta$, then we `cap' the distortion by increasing the measure at each point of $F_x \setminus B_k$ by a factor of $1 / (1-\delta)$, and decreasing the measure on points of $F_x \cap B_k$ by a corresponding factor.
\end{itemize}
To be precise, the probability measure $\Pr_k$ is defined as follows.

\begin{defn}\label{def:Pk}
For each $(x,y) \in Q_k$, define
\begin{equation*}\label{def:P_k}
 \Pr_k(x,y) :=
 \begin{cases}
   \max\bigg\{ 0, \, \ds\frac{\alpha_k(x)-\delta}{\alpha_k(x)(1-\delta)} \bigg\} \cdot \frac{\Pr_{k-1}(x)}{|S_k|},
   & \text{ if }(x,y)\in B_k;\\[3ex]
   \min\bigg\{ \ds\frac{1}{1-\alpha_k(x)}, \, \frac{1}{1-\delta} \bigg\} \cdot \frac{\Pr_{k-1}(x)}{|S_k|},
   & \text{ if }(x,y)\notin B_k.
 \end{cases}
\end{equation*}
\end{defn}

Note that $\sum_{y \in S_k} \Pr_k(x,y) = \Pr_{k-1}(x)$ for every $x \in Q_{k-1}$, and hence $\Pr_k(X) = \Pr_{k-1}(X)$ for any $Q_{k-1}$-measurable set $X$. We can now easily prove the following key lemma, which (despite its simplicity) is the main step in the proof of Theorem~\ref{thm:GMM}.


\begin{lemma}\label{thm:general}
Let\/ $\cA$ be a collection of hyperplanes in $Q = S_1 \times \dots\times S_n$. If
 \begin{equation}\label{eq:eta}
 \frac{1}{4\delta(1-\delta)} \sum_{k = 1}^n \E_{k-1}\big[\alpha_k(x)^2\big] < 1,
 \end{equation}
 then\/ $\cA$ does not cover\/~$Q$.
\end{lemma}

\begin{proof}
Recall from~\eqref{def:alpha} that $|F_x \cap B_k| = \alpha_k(x) \cdot |S_k|$. By Definition~\ref{def:Pk}, it follows that
\begin{align*}
\Pr_k(B_k) & = \sum_{x \in Q_{k-1}} | F_x \cap B_k| \cdot \max\bigg\{ 0, \, \frac{\alpha_k(x) - \delta}{\alpha_k(x)( 1 - \delta )} \bigg\} \cdot \frac{\Pr_{k-1}(x)}{|S_k|} \nonumber\\
& = \frac{1}{1-\delta} \sum_{x \in Q_{k-1}} \max\big\{ 0, \, \alpha_k(x)-\delta \big\} \cdot \Pr_{k-1}(x) \nonumber\\
& \le \frac{1}{1-\delta}\sum_{x \in Q_{k-1}} \frac{\alpha_k(x)^2}{4\delta} \cdot \Pr_{k-1}(x)
\, = \, \frac{\E_{k-1}\big[ \alpha_k(x)^2 \big]}{4\delta(1-\delta)}. \label{eq:bound:on:Bk}
\end{align*}
Indeed, $\max\{a-b,0\} \le a^2 / 4b$ follows from $(a-2b)^2 \ge 0$, and holds for all $a,b > 0$.

Now, since $\Pr_n(B_k) = \Pr_k(B_k)$ for every $1 \le k \le n$, it follows that the set $R \subset Q$ of points not covered by $\cA$ satisfies
$$\Pr_n(R) \ge 1 - \sum_{k = 1}^n \Pr_n(B_k) \ge 1 - \frac{1}{4\delta(1-\delta)} \sum_{k = 1}^n \E_{k-1}\big[\alpha_k(x)^2\big] > 0,$$
by~\eqref{eq:eta}, and hence $\cA$ does not cover $Q$, as claimed.
\end{proof}

\section{Bounding the moments of $\alpha_k(x)$}\label{sec:moments}

In order to use Lemma~\ref{thm:general} to prove Theorem~\ref{thm:GMM}, we need to bound, for each $1 \le k \le n$, the second moment of~$\alpha_k(x)$ with respect to the measure $\Pr_{k-1}$. The following lemma provides the bound we need.

\begin{lemma}\label{prop:moments}
Let\/ $\cA$ be a collection of hyperplanes in $Q = S_1 \times \dots\times S_n$, no two of which are parallel. Then, for each\/ $1 \le k \le n$,
\begin{equation}\label{eq:M:bounds}
\E_{k-1}\big[\alpha_k(x)^2\big] \le \frac{1}{|S_k|^2} \prod_{j = 1}^{k-1} \bigg(1 + \frac{3}{(1-\delta)|S_j|} \bigg). 
\end{equation}
\end{lemma}

The first step in the proof of Lemma~\ref{prop:moments} is the following straightforward bound on the $\Pr_k$-measure of a $Q_k$-measurable hyperplane.

\begin{lemma}\label{lem:AP}
Let\/ $A$ be a hyperplane, and let $0 \le k \le n$. If $F(A) \subset [k]$, then
\begin{equation}\label{e:APdist}
  \Pr_k(A) \le \prod_{j \in F(A)} \frac{1}{(1-\delta)|S_j|}. 
\end{equation}
\end{lemma}

In the proof of Lemma~\ref{lem:AP} we will use the following simple properties of the measures $\Pr_k$. Recall from Section~\ref{sec:distortion} that
 \begin{equation}\label{obs:Qmeas}
  \Pr_k(X) = \Pr_{k-1}(X)
 \end{equation}
for any $Q_{k-1}$-measurable set $X$, and observe that
 \begin{equation}\label{e:dom}
 \Pr_k(X) \le \frac{1}{1-\delta} \cdot \Pr_{k-1}(X)
 \end{equation}
for any set $X \subset Q$, by Definition~\ref{def:Pk}. We will find it useful to define
\begin{equation*}\label{def:nu}
\nu(J) := \prod_{j \in J} \frac{1}{(1-\delta)|S_j|}
\end{equation*}
for each $J \subset [n]$ and, given a hyperplane $A = Y_1 \times \cdots \times Y_n$ and a set $U \subset [n]$, to define $A^U := Y^U_1 \times \cdots Y^U_n$ to be the hyperplane with $Y^U_i := Y_i$ if $i \in U$, and $Y^U_i := S_i$ otherwise.

\begin{proof}[Proof of Lemma~\ref{lem:AP}]
We will prove, by induction on~$k$, that $\Pr_k(A) \le \nu(J)$ for all $0 \le k \le n$, every set $J \subset [k]$, and every hyperplane $A$ with $F(A) = J$. For $k = 0$ this follows because $\nu(\emptyset) = 1$, so let $1 \le k \le n$, and assume that the induction hypothesis holds for~$\Pr_{k-1}$.

Suppose first that $k \not\in F(A)$. Then $A$ is $Q_{k-1}$-measurable and $J \subset [ k-1]$, and it follows by~\eqref{obs:Qmeas} and the induction hypothesis that $\Pr_k(A) = \Pr_{k-1}(A) \le \nu(J)$, as required.

\enlargethispage*{\baselineskip}

On the other hand, if $k \in F(A)$, then it follows from~\eqref{e:dom} that
$$\Pr_k( A ) \le \frac{1}{1-\delta} \cdot \Pr_{k-1}( A ) = \frac{1}{(1-\delta)|S_k|} \cdot \Pr_{k-1}\big( A^{[k-1]} \big),$$
since the probability measure $\Pr_{k-1}$ is extended uniformly on each fibre. Since $F(A^{[k-1]}) = J \setminus \{k\} \subset [k-1]$, it follows from the induction hypothesis that
$$\Pr_{k-1}\big( A^{[k-1]} \big) \le \nu( J \setminus \{k\}).$$
Hence, by the definition of $\nu$, we obtain $\Pr_k( A ) \le \nu(J)$, as claimed.
\end{proof}

Using Lemma~\ref{lem:AP}, we can now prove the following bound on the second moment of $\alpha_k(x)$. 

\begin{lemma}\label{lem:moments}
Let\/ $\cA$ be a collection of hyperplanes in $Q = S_1 \times \dots\times S_n$, no two of which are parallel. Then, for each\/ $1 \le k \le n$,$$\E_{k-1}\big[ \alpha_k(x)^2 \big] \le \frac{1}{|S_k|^2} \sum_{F_1, F_2 \subset [k-1]} \nu\big( F_1 \cup F_2 \big).$$
\end{lemma}

\begin{proof}
Recalling the definitions of~$\alpha_k$ and $B_k$, and using the union bound, we obtain
$$\alpha_k(x) \, = \frac{1}{|S_k|} \sum_{y \in S_k} \id\big[ (x,y) \in B_k \big] \le \frac{1}{|S_k|} \sum_{y \in S_k} \sum_{A \in \cA_k} \id\big[ (x,y) \in A \big]$$
for each $x \in Q_{k-1}$, and therefore
$$\alpha_k(x) \le \frac{1}{|S_k|} \sum_{A \in \cA_k} \id\big[ x \in A^{[k-1]} \big],$$
since for each $x \in Q_{k-1}$ and $A \in \cA_k$, there exists $y \in S_k$ with $(x,y) \in A$ if and only if $x \in A^{[k-1]}$, and moreover such a $y$ (if it exists) is unique, since $k \in F(A)$. It follows that
$$\E_{k-1}\big[ \alpha_k(x)^2 \big] \le \frac{1}{|S_k|^2} \sum_{A_1, A_2 \in \cA_k} \Pr_{k-1}\big( A_1^{[k-1]} \cap A_2^{[k-1]} \big).$$

Now, if $A_1^{[k-1]} \cap A_2^{[k-1]}$ is non-empty, then it is a hyperplane whose set of fixed coordinates is $F_1 \cup F_2$, where $F_1 = F(A_1) \cap [k-1]$ and $F_2 = F(A_2) \cap [k - 1]$. Moreover, the sets $F_i$ determine the hyperplanes $A_i\in \cA_k$ uniquely, since no two of the hyperplanes of $\cA$ are parallel. Hence, applying Lemma~\ref{lem:AP} and recalling the definition of $\nu$, it follows that
$$\E_{k-1}\big[ \alpha_k(x)^2 \big] \le \frac{1}{|S_k|^2} \sum_{F_1, F_2 \subset [k-1]} \nu\big( F_1 \cup F_2 \big),$$
as required.
\end{proof}

The claimed bound on $\E_{k-1}\big[\alpha_k(x)^2\big]$ now follows easily.

\begin{proof}[Proof of Lemma~\ref{prop:moments}]
Observe that
$$\sum_{F_1, F_2 \subset [k-1]} \nu\big( F_1 \cup F_2 \big) = \sum_{J \subset [k-1]} \sum_{\substack{F_1, F_2 \subset [k-1] \\ F_1 \cup F_2 = J}} \nu(J) = \sum_{J \subset [k-1]} 3^{|J|} \nu(J).$$
Hence, by Lemma~\ref{lem:moments},  and recalling again the definition of $\nu$, we have
$$\E_{k-1}\big[\alpha_k(x)^2 \big] \le \frac{1}{|S_k|^2} \sum_{J \subset [k-1]} 3^{|J|} \nu(J) \, = \, \frac{1}{|S_k|^2} \prod_{j = 1}^{k-1} \bigg(1 + \frac{3}{(1-\delta)|S_j|} \bigg),$$
as required.
\end{proof}

\section{The proof of Theorem~\ref{thm:GMM}}\label{sec:proofGMM}

Theorem~\ref{thm:GMM} is a straightforward consequence of Lemmas~\ref{thm:general} and~\ref{prop:moments}; we just need to choose $C$ and $\delta$ so that if $F(A) \not\subset \{1,\ldots,C\}$ for every $A \in \cA$, then the bound given by Lemma~\ref{prop:moments} is strong enough to imply that~\eqref{eq:eta} holds.

\begin{proof}[Proof of Theorem~\ref{thm:GMM}]
Let $(S_k)_{k \ge 1}$ be a sequence of sets as in the statement of the theorem, so there exist $N \in \N$ and $0 < \eps \le 1$ such that $|S_k| \ge (3 + \eps)k$ for all $k \ge N$, and moreover $|S_k| \ge 2$ for each $k \in \N$. We will show that if $C = C(N,\eps)$ is sufficiently large, then the conclusion of the theorem holds. Let $\cA$ be a collection of hyperplanes in $Q = S_1 \times \dots \times S_n$, no two of which are parallel, and with $F(A) \not\subset \{1,\ldots,C\}$ for every $A \in \cA$. To prove the theorem it will suffice to show that $\cA$ does not cover $Q$.

Set $\delta := \eps/6 \in (0,1/2]$, and observe that $\alpha_k(x) = 0$ for every $1 \le k \le C$ and $x \in Q_{k-1}$, since $F(A) \not\subset \{1,\ldots,C\}$ for every $A \in \cA$. Moreover, by Lemma~\ref{prop:moments},
$$\E_{k-1}\big[\alpha_k(x)^2\big] \le \frac{1}{|S_k|^2} \prod_{j = 1}^{k-1} \bigg(1 + \frac{3}{(1-\delta)|S_j|} \bigg)$$
for each $C < k \le n$. Now, note that $(1-\delta)|S_j| \ge 1$ for every $j \in \N$, and that if $j \ge N$ then $(1-\delta)|S_j| \ge (1 - \eps/6)(3 + \eps) \cdot j$. Thus
$$\prod_{j = 1}^{k-1} \bigg(1 + \frac{3}{(1-\delta)|S_j|} \bigg) \le 4^N \exp\bigg( \frac{3}{(1 - \eps/6)(3 + \eps)} \sum_{j = N}^{k-1} \frac{1}{j} \bigg) \le 4^N  \cdot k^{1-\eps/10},$$
where the final inequality holds since $\sum_{j = N}^{k-1} 1/j \le \log k$ and $(1-\eps/6)(3 + \eps)(1 - \eps/10) \ge 3$.

It follows that
$$\E_{k-1}\big[\alpha_k(x)^2\big] \le \, \frac{4^N}{|S_k|^2} \cdot k^{1-\eps/10} \le \, \frac{4^{N}}{9\cdot k^{1+\eps/10}}$$
for every $C < k \le n$ (as long as we chose $C \ge N$, so that $|S_k| \ge 3k$), and hence
$$\frac{1}{4\delta(1-\delta)} \sum_{k=1}^n \E_{k-1}\big[\alpha_k(x)^2\big] \le \, \frac{4^N}{\eps} \sum_{k = C}^n \frac{1}{k^{1+\eps/10}} < \, 1$$
if $C = C(N,\eps)$ is sufficiently large. By Lemma~\ref{thm:general} it follows that $\cA$ does not cover $Q$, as required.
\end{proof}

\addtolength{\footskip}{-\baselineskip/2}

\begin{remark}
When $|S_k| = p_k$, the $k$th prime, for each $k \in \N$, we can choose $\eps = 1$ and $N = 31$, in which case the final inequality in the proof above holds as long as $C \ge 10^{200}$. By the proof of Corollary~\ref{cor:MM}, this gives a (fairly terrible) bound of roughly $\exp( 10^{200} )$ for the minimum modulus in a covering system with distinct square-free moduli. However, it is clear that one could do rather better with a little more effort, and in~\cite{BBMST1} we used a variant of the proof above to reduce the bound in Hough's theorem to less than $10^6$.
\end{remark}

\pagebreak

\section*{Acknowledgement}

The authors would like to thank Noga Alon for an interesting conversation that motivated us to write this expository note.

\end{document}